\documentclass[a4paper,12pt]{article}
\usepackage[top=2.5cm,bottom=2.5cm,left=2.5cm,right=2.5cm]{geometry}
\usepackage{cite, amsmath, amssymb}
\usepackage[margin=1cm,%
                font=small,%
                format=hang,%
                labelsep=period,%
                labelfont=bf]{caption}
\usepackage[latin1]{inputenc}
\usepackage{amsmath}
\usepackage{amsfonts}
\usepackage{amssymb}
\usepackage{cite}
\usepackage{amsthm}
\usepackage{makeidx}
\usepackage{graphicx}
\usepackage{mathrsfs,xcolor}
\usepackage{enumerate}
\usepackage{blkarray}
\usepackage{bm}
\usepackage{setspace}
\usepackage{float}
\usepackage[ruled,vlined]{algorithm2e}
\usepackage{authblk}

\numberwithin{table}{section}
\numberwithin{equation}{section}

\theoremstyle{plain}
\newtheorem{theorem}{Theorem}[section]
\newtheorem{proposition}[theorem]{Proposition}
\newtheorem{definition}[theorem]{Definition}

\newtheorem{example}[theorem]{Example}

\newtheorem{corollary}[theorem]{Corollary}
\newtheorem{remark}[theorem]{Remark}

\usepackage{authblk}
\author[1,*]{ \textbf{Bryan S. Hernandez}}
\author[1]{ \textbf{Deza A. Amistas}}
\author[1]{\textbf{Ralph John L. De la Cruz}}
\author[2,3]{ \textbf{Lauro L. Fontanil}}
\author[1]{\textbf{Aurelio A. de los Reyes V}}
\author[3,4,5,6]{\textbf{Eduardo R. Mendoza}}

\affil[1]{\small \textit{Institute of Mathematics, University of the Philippines Diliman, Quezon City 1101, Philippines}}
\affil[2]{\small \textit{Institute of Mathematical Sciences and Physics, University of the Philippines Los Ba{\~n}os, Laguna 4031, Philippines}}
\affil[3]{\small \textit{Mathematics and Statistics Department, De La Salle University, Manila  0922, Philippines}}
\affil[4]{\small \textit{Center for Natural Sciences and Environmental Research, De La Salle University, Manila 0922, Philippines}}
\affil[5]{\small \textit{Max Planck Institute of Biochemistry, Martinsried, Munich, Germany}}
\affil[6]{\small \textit{LMU Faculty of Physics, Geschwister -Scholl- Platz 1, 80539 Munich, Germany}}
\affil[*]{Corresponding author: \texttt{bryan.hernandez@upd.edu.ph}}

\title{\textbf{Independent, Incidence Independent and Weakly Reversible Decompositions of Chemical Reaction Networks}}

\date{}

\begin{document}
\maketitle
\thispagestyle{empty}
\begin{abstract}
Chemical reaction networks (CRNs) are directed graphs with reactant or product complexes as vertices, and reactions as arcs.
A CRN is weakly reversible if each of its connected components is strongly connected. Weakly reversible networks can be considered as the most important class of reaction networks.
Now, the stoichiometric subspace of a network is the linear span
of the reaction vectors (i.e., difference between the product and the reactant complexes).
A decomposition of a CRN is independent (incidence independent) if the direct sum of the stoichiometric subspaces (incidence maps) of the subnetworks equals the stoichiometric subspace (incidence map) of the whole network.
Decompositions can be used to study relationships between steady states of the whole system (induced from partitioning the reaction set of the underlying network) and those of its subsystems. 
In this work, we revisit our novel method of finding independent decomposition, and use it to expand applicability on (vector) components of steady states. We also explore CRNs with embedded deficiency zero independent subnetworks. In addition, we establish a method for finding incidence independent decomposition of a CRN.
We determine all the forms of independent and incidence independent decompositions of a network, and provide the number of such decompositions.
Lastly, for weakly reversible networks, we determine that incidence independence is a sufficient condition for weak reversibility of a decomposition, and we identify subclasses of weakly reversible networks where any independent decomposition is weakly reversible.
\end{abstract}
\baselineskip=0.30in

\section{Introduction}
\label{sec:1}
A chemical reaction network (CRN) is a finite set of reactions. It is a directed graph with reactant or product complexes as vertices, and reactions as arcs.
A reaction vector of a reaction is the difference between the product and the reactant complexes.
In addition, the stoichiometric subspace of a CRN is the linear span, over $\mathbb{R}$, of the reaction vectors.
A decomposition of a CRN is independent (incidence independent) if the direct sum of the stoichiometric subspaces (incidence maps) of the subnetworks equals the stoichiometric subspace (incidence map) of the whole network.

M. Feinberg introduced an important result that provides a relationship between the sets of positive steady states of a kinetic system and its subsystems induced by decomposition of the underlying reaction network \cite{feinberg12,feinberg:book}.
In particular, for an independent decomposition, the intersection of the sets of positive steady states of the subsystems is equal to the set of positive steady states of the whole system. This is one of the reasons why we are interested with independent decompositions.
One may decompose large networks and analyze smaller ones. Recently, the paper of Hernandez and De la Cruz in \cite{hernandez:delacruz1} provided a novel method of finding an independent decomposition of a CRN.
On the other hand, Fari\~nas et al. provided their analogue for incidence independent decompositions, which has something to do with complex balanced steady states \cite{FAML2020}.

Both independent and incidence independent decompositions of chemical reaction networks can be observed in literature. Such ubiquitous occurrences include networks of phosphorylation/dephosphorylation systems \cite{CODH2018,hmr22019}, virus models \cite{virus,hernandez:delacruz1}, carbon cycle models \cite{schmitz,fortun2,hmr22019},  and gene regulatory systems of {\textit{Mycobacterium tuberculosis} \cite{mm,mtb}.
	
	A CRN is weakly reversible if every pair of vertices in every linkage class is strongly connected. Weakly reversible CRNs have been the focus of many CRNT studies \cite{boros2020,boros2019,boros2013,cracium2020,feinberg,HM2021:WRDCKS,mendoza2018,muller2012,talabis2020,talabis2019}. Some of these highlight the  importance of such CRNs in the studies of complex balanced equilibria and positive steady  states. One of the earliest to examine weakly reversible CRNs is M. Feinberg. He established that it is sufficient and necessary for a complex balanced CRN to be weakly reversible \cite{feinberg}. In the same paper, he has shown  that a deficiency zero chemical kinetics system that is not weakly reversible does not have a positive steady state.
	
	The goal of this work is to explore decomposition theory of reaction networks, and study properties of steady states of the corresponding whole system using properties of steady states of its subsystems.
	In particular, we aim to do the following:
	\begin{itemize}
		\item[1.] revisit the method of Hernandez and De la Cruz of finding an independent decomposition of a CRN and provide results for networks with embedded deficiency zero subnetwork,
		\item[2.] establish a method of finding incidence independent decomposition of a CRN,
		\item[3.] determine all the forms of independent and incidence independent decompositions,
		\item[4.] provide the number of independent and incidence independent decompositions of a CRN, and
		\item[5.] give connections between weakly reversible decompositions and the two previous decompositions.
	\end{itemize}
	
	The paper is organized as follows: Section \ref{prelim} gives some fundamental results on chemical reaction networks, chemical kinetic systems, and decomposition theory that are needed in this work. Section \ref{section:results1} provides important results on the independent and incidence independent decompositions, including algorithm, counting, and applications.
	Section \ref{section:results3} deals with conditions about weakly reversible decompositions through independent and incidence independent decompositions. Finally, summary and outlook are given in Section \ref{section:summary}.

	\section{Preliminaries}
	\label{prelim}
	
	This section provides a discussion of fundamental concepts on chemical reaction networks, chemical kinetic systems and essential results on network decomposition theory \cite{arceo2015,feinberg12,feinberg}.
	
	\subsection{Fundamentals of Chemical Reaction Networks}
	
	We formally define a chemical reaction network (CRN) and other notions important to our work.
	
	\begin{definition}
		A {\bf chemical reaction network} $\mathscr{N}$ is given by $\left(\mathscr{S},\mathscr{C},\mathscr{R}\right)$ of nonempty finite sets $\mathscr{S}$, $\mathscr{C} \subseteq \mathbb{R}_{\ge 0}^\mathscr{S}$, and $\mathscr{R} \subset \mathscr{C} \times \mathscr{C}$, of $m$ species, $n$ complexes, and $r$ reactions, respectively, such that the following are satisfied:
		\begin{itemize}
			\item[i.] $\left( {{C_i},{C_i}} \right) \notin \mathscr{R}$ for each $C_i \in \mathscr{C}$, and
			\item[ii.]  for each $C_i \in \mathscr{C}$, there exists $C_j \in \mathscr{C}$ such that $\left( {{C_i},{C_j}} \right) \in \mathscr{R}$ or $\left( {{C_j},{C_i}} \right) \in \mathscr{R}$.
		\end{itemize}
	\end{definition}
	
	We now define the following matrices: molecularity matrix, incidence matrix, and stoichiometric matrix, which we denote by $Y$, $I_a$, and $N$, respectively.
	\begin{itemize}
		\item[1.] The {\bf molecularity matrix} is an $m\times n$ matrix such that $Y_{ij}$ is the stoichiometric coefficient (i.e., scalar that appear in front) of species $X_i$ in complex $C_j$.
		\item[2.] The {\bf incidence matrix} is an $n\times r$ matrix such that 
		$${\left( {{I_a}} \right)_{ij}} = \left\{ \begin{array}{rl}
			- 1&{\rm{ if \ }}{C_i}{\rm{ \ is \ in \ the\ reactant \ complex \ of \ reaction \ }}{R_j},\\
			1&{\rm{  if \ }}{C_i}{\rm{ \ is \ in \ the\ product \ complex \ of \ reaction \ }}{R_j},\\
			0&{\rm{    otherwise}}.
		\end{array} \right.$$
		\item[3.] The {\bf stoichiometric matrix} is the $m\times r$ matrix given by 
		$N=YI_a$.
	\end{itemize}

	We denote the standard basis for $\mathbb{R}^\mathscr{I}$ by $\left\lbrace \omega_i \in \mathbb{R}^\mathscr{I} \mid i \in \mathscr{I} \right\rbrace$.
	For a CRN $\mathscr{N}=(\mathscr{S,C,R})$, the {\bf{incidence map}} $I_a : \mathbb{R}^\mathscr{R} \rightarrow \mathbb{R}^\mathscr{C}$ is the linear map such that for each reaction $r:C_i \rightarrow C_j \in \mathscr{R}$, the basis vector $\omega_r$ is mapped to the vector $\omega_{C_j}-\omega_{C_i} \in \mathscr{C}$.
	The {\bf reaction vectors} of a network are the elements of the set $\left\{{C_j} - {C_i} \in \mathbb{R}^m|\left( {{C_i},{C_j}} \right) \in \mathscr{R}\right\}.$
	The {\bf stoichiometric subspace} of the network is the linear subspace of $\mathbb{R}^m$ given by $$S = {\rm{span}}\left\{ {{C_j} - {C_i} \in \mathbb{R}^m|\left( {{C_i},{C_j}} \right) \in \mathscr{R}} \right\}.$$ The {\bf rank} of the network is given by $s=\dim S$. The set $\left( {x + S} \right) \cap \mathbb{R}_{ \ge 0}^m$ is said to be a {\bf stoichiometric compatibility class} of $x \in \mathbb{R}_{ \ge 0}^m$.
	In addition, two vectors $u, v \in {\mathbb{R}^m}$ are {\bf stoichiometrically compatible} if $u-v \in S$.
	
	CRNs are directed graphs where complexes are vertices and reactions are arcs. Two vertices are {\bf connected} if there is a directed path between them.
	They are {\bf strongly connected} if there is a directed path from one vertex to the other, and vice versa. 
	A subgraph is a {\bf (strongly) connected component}
	if any two vertices of the subgraph are {\bf (strongly) connected}.
	The {\bf (strong) linkage classes} of a CRN are the (strong) connected components of the graph. 
	A {\bf terminal strong linkage classes} is a
	maximal strongly connected subgraph where there are no edges from a complex in the subgraph to a complex outside the subgraph.
	We denote the number of linkage classes by $l$, and the number of strong linkage classes by $sl$.
	A CRN is {\bf weakly reversible} if $sl=l$, i.e., each linkage class is a strong linkage class. In particular, a CRN is {\bf reversible} if each reaction is reversible, i.e., for any reaction $y' \to y$, there exists $y \to y' \in \mathscr{R}$.
	
	\begin{definition}
		The {\bf deficiency} of a CRN is given by $$\delta=n-l-s$$ where $n$ is the number of complexes, $l$ is the number of linkage classes, and $s$ is the dimension of the stoichiometric subspace $S$.
	\end{definition}
	
	\begin{example}
		\label{running:example}
		Consider the CRN of the Baccam model (an influenza virus model) with delayed virus production \cite{baccam,miao}. It has four variables: susceptible target cells denoted by $T$, two types of infected cells, namely, infected cells but not yet producing virus denoted by $I_1$ and infected cells that actively producing virus denoted by $I_2$, and infectious-viral titer denoted by $V$. The following are the reactions \cite{virus}.
		\[ \begin{array}{lll}
			R_1: T+V \to I_1+V & \ \ \ &  R_4: I_2 \to I_2 + V\\
			R_2: I_1 \to I_2  &  \ \ \ & R_5: V \to 0 \\
			R_3: I_2 \to 0 &  \ \ \ & \\
		\end{array}\]
		We compute for the molecularity, incidence, and stoichiometric matrices:
		\begin{align*}
			Y&=
			\begin{blockarray}{cccccccc}
				&T+V & I_1+V & I_1 & I_2 & 0 & I_2+V & V\\
				\begin{block}{c[ccccccc]}
					T & 1 & 0 & 0 & 0 & 0 & 0 & 0\\
					V & 1 & 1 & 0 & 0 & 0 & 1 & 1\\
					I_1 & 0 & 1 & 1 & 0 & 0 & 0 & 0\\
					I_2 & 0 & 0 & 0 & 1 & 0 & 1 & 0\\
				\end{block}
			\end{blockarray},\\
			I_a&=
			\begin{blockarray}{cccccc}
				&R_1 & R_2 & R_3 & R_4 & R_5 \\
				\begin{block}{c[ccccc]}
					T+V & -1 & 0 & 0 & 0 & 0 \\
					I_1+V & 1 & 0 & 0 & 0 & 0 \\
					I_1 & 0 & -1 & 0 & 0 & 0 \\
					I_2 & 0 & 1 & -1 & -1 & 0 \\
					0 & 0 & 0 & 1 & 0 & 1 \\
					I_2+V & 0 & 0 & 0 & 1 & 0 \\
					V & 0 & 0 & 0 & 0 & -1 \\
				\end{block}
			\end{blockarray},
	\end{align*}
	\begin{align*}
			N&=
			\begin{blockarray}{cccccc}
				&R_1 & R_2 & R_3 & R_4 & R_5 \\
				\begin{block}{c[ccccc]}
					T & -1 & 0 & 0 & 0 & 0 \\
					V & 0 & 0 & 0 & 1 & -1 \\
					I_1 & 1 & -1 & 0 & 0 & 0 \\
					I_2 & 0 & 1 & -1 & 0 & 0 \\
				\end{block}
			\end{blockarray}.
		\end{align*}
	\end{example}
	
	The number of complexes is $n=7$, there are two linkage classes, i.e., $l=2$. Also, the rank of the network is $s=4$. Hence, the deficiency of the network is $\delta = n-l-s=7-2-4=1$. The network is not weakly reversible since $sl=5$.
	
	\subsection{Fundamentals of Chemical Kinetic Systems}
	
	\begin{definition}
		A {\bf kinetics} $K$ for a reaction network $\left(\mathscr{S},\mathscr{C},\mathscr{R}\right)$ is an assignment to each reaction $r: y \to y' \in \mathscr{R}$ of a rate function ${K_r}:{\Omega _K} \to {\mathbb{R}_{ \ge 0}}$ such that $\mathbb{R}_{ > 0}^m \subseteq {\Omega _K} \subseteq \mathbb{R}_{ \ge 0}^m$, $c \wedge d \in {\Omega _K}$ if $c,d \in {\Omega _K}$, and ${K_r}\left( c \right) \ge 0$ for each $c \in {\Omega _K}$.
		Furthermore, it satisfies the positivity property: supp $y$ $\subset$ supp $c$ if and only if $K_r(c)>0$ (supp $y$ is the support of vector $y$).
		The system $\left(\mathscr{S},\mathscr{C},\mathscr{R},K\right)$ is called a {\bf chemical kinetic system}.
	\end{definition}
	
	\begin{definition}
		The {\bf species formation rate function} (SFRF) of a chemical kinetic system is given by $f\left( x \right) = NK(x)= \displaystyle \sum\limits_{{C_i} \to {C_j} \in \mathscr{R}} {{K_{{C_i} \to {C_j}}}\left( x \right)\left( {{C_j} - {C_i}} \right)}.$
	\end{definition}
	
	The ordinary differential equation (ODE) or dynamical system of a chemical kinetics system is $\dfrac{{dx}}{{dt}} = f\left( x \right)$. An {\bf equilibrium} or {\bf steady state} is a zero of $f$.
	The {\bf set of positive equilibria} of a CKS $\left(\mathscr{S},\mathscr{C},\mathscr{R},K\right)$ is given by $${E_ + }\left(\mathscr{S},\mathscr{C},\mathscr{R},K\right)= \left\{ {x \in \mathbb{R}^m_{>0}|f\left( x \right) = 0} \right\}.$$
	A CRN is said to admit {\bf multiple (positive) equilibria} if there exist positive rate constants such that the ODE system admits more than one stoichiometrically compatible equilibria.
	Analogously, the {\bf set of complex balanced equilibria} \cite{HornJackson} is given by 
	\[{Z_ + }\left(\mathscr{N},K\right) = \left\{ {x \in \mathbb{R}_{ > 0}^m|{I_a} \cdot K\left( x \right) = 0} \right\} \subseteq {E_ + }\left(\mathscr{N},K\right).\]
	A positive vector $c \in \mathbb{R}^m$ is complex balanced if $K\left( c \right)$ is contained in ${\text{Ker }}{I_a}$, and a chemical kinetic system is {\bf complex balanced} if it has a complex balanced equilibrium.
	
	\subsection{Review of Decomposition Theory}
	\label{sect:decomposition}
	We formally define a decomposition of a CRN: 
	\begin{definition}
		A {\textbf{decomposition}} of $\mathscr{N}$ is a set of subnetworks $\{\mathscr{N}_1, \mathscr{N}_2,...,\mathscr{N}_k\}$ of $\mathscr{N}$ induced by a partition $\{\mathscr{R}_1, \mathscr{R}_2,...,\mathscr{R}_k\}$ of its reaction set $\mathscr{R}$. 
	\end{definition}
	
	We shall call the decomposition of a network with a single subnetwork as the {\bf{trivial decomposition}}.
	
	We denote a network decomposition by 
	$\mathscr{N} = \mathscr{N}_1 \cup \mathscr{N}_2 \cup ... \cup \mathscr{N}_k$
	as used in \cite{ghms2019}. Hence, for the corresponding stoichiometric subspaces, 
	$${S} = {S}_1 + {S}_2 + ... + {S}_k.$$
	A network decomposition $\mathscr{N} = \mathscr{N}_1 \cup \mathscr{N}_2 \cup ... \cup \mathscr{N}_k$  is a {\bf refinement} of
	$\mathscr{N} = {\mathscr{N}'}_1 \cup {\mathscr{N}'}_2 \cup ... \cup {\mathscr{N}'}_{k'}$ 
	(and the latter a {\bf coarsening} of the former) if it is induced by a refinement  
	$\{\mathscr{R}_1, \mathscr{R}_2,...,\mathscr{R}_k\}$
	of $\{{\mathscr{R}'}_1 \cup {\mathscr{R}'}_2 \cup ... \cup {\mathscr{R}'}_{k'}\}$.
	
	We are now ready to define independent and incidence independent decompositions \cite{feinberg12}.
	
	\begin{definition}
		A network decomposition $\mathscr{N} = \mathscr{N}_1 \cup \mathscr{N}_2 \cup ... \cup \mathscr{N}_k$  is said to be {\textbf{independent}} if its stoichiometric subspace is a direct sum of the subnetwork stoichiometric subspaces.
		It is said to be {\textbf{incidence independent}} if the incidence map $I_a$ of $\mathscr{N}$ is the direct sum of the incidence maps of the subnetworks. It is said to be {\textbf{bi-independent}} if it is both independent and incidence independent.
	\end{definition}
	
	We can also show incidence independence by satisfying the equation $$n - l = \sum {\left( {{n_i} - {l_i}} \right)},$$ where $n_i$ is the number of complexes and $l_i$ is the number of linkage classes, in each subnetwork $i$.
	It was given in \cite{fortun2}, that for an independent decomposition, $$\delta \le \delta_1 +\delta_2+ \cdots +\delta_k.$$
	On the other hand, it was shown in \cite{FAML2020} that for an incidence independent decomposition, $$\delta \ge \delta_1 +\delta_2+ \cdots +\delta_k.$$
	Thus, for bi-independent decompositions, the following equation holds:
	$$\delta = \delta_1 +\delta_2+ \cdots +\delta_k.$$
	
	\begin{example}
		\label{running:example3}
		We again refer to Example \ref{running:example}, the Baccam CRN with the following reactions:
	\[ \begin{array}{lll}
		R_1: T+V \to I_1+V & \ \ \ &  R_4: I_2 \to I_2 + V\\
		R_2: I_1 \to I_2  &  \ \ \ & R_5: V \to 0 \\
		R_3: I_2 \to 0 &  \ \ \ & \\
	\end{array}\]
		In \cite{hernandez:delacruz1}, it was shown that $\{\mathscr{N}_1,\mathscr{N}_2\}$, where $\mathscr{N}_1=\{R_1,R_2,R_3\}$ and $\mathscr{N}_2=\{R_4,R_5\}$,
		is an independent decomposition. We can also check that the decomposition is also incidence independent. Thus, the given decomposition of the CRN is bi-independent.
	\end{example}
	
	The following theorems in Decomposition Theory of CRNs provide important relationships between independent and incidence independent decompositions, and the set of positive equilibria of the given kinetic system. Theorem \ref{feinberg:decom:thm} was a result by M. Feinberg, which we call the Feinberg Decomposition Theorem \cite{feinberg12}. 
	This was also discussed in pages 84-88 of his book in 2019 given in \cite{feinberg:book}.
	In an analogous manner, Fari\~nas et al. introduced their results for incidence independent decompositions and complex balanced equilibria \cite{FAML2020}.
	
	\begin{theorem} (\cite{feinberg12,feinberg:book})
		\label{feinberg:decom:thm}
		Let $P(\mathscr{R})=\{\mathscr{R}_1, \mathscr{R}_2,...,\mathscr{R}_k\}$ be a partition of a CRN $\mathscr{N}$ and let $K$ be a kinetics on $\mathscr{N}$. If $\mathscr{N} = \mathscr{N}_1 \cup \mathscr{N}_2 \cup ... \cup \mathscr{N}_k$ is the network decomposition of $P(\mathscr{R})$ and ${E_ + }\left(\mathscr{N}_i,{K}_i\right)= \left\{ {x \in \mathbb{R}^\mathscr{S}_{>0}|N_iK_i(x) = 0} \right\}$ then
		\[{E_ + }\left(\mathscr{N}_1,K_1\right) \cap {E_ + }\left(\mathscr{N}_2,K_2\right) \cap ... \cap {E_ + }\left(\mathscr{N}_k,K_k\right) \subseteq  {E_ + }\left(\mathscr{N},K\right).\]
		If the network decomposition is independent, then equality holds.
	\end{theorem}
	
	\begin{theorem} (Theorem 4 \cite{FAML2020})
		\label{decomposition:thm:2}
		Let $\mathscr{N}=(\mathscr{S},\mathscr{C},\mathscr{R})$ be a a CRN and $\mathscr{N}_i=(\mathscr{S}_i,\mathscr{C}_i,\mathscr{R}_i)$ for $i = 1,2,...,k$ be the subnetworks of a decomposition.
		Let $K$ be any kinetics, and $Z_+(\mathscr{N},K)$ and $Z_+(\mathscr{N}_i,K_i)$ be the sets of complex balanced equilibria of $\mathscr{N}$ and $\mathscr{N}_i$, respectively. Then
		\begin{itemize}
			\item[i.] ${Z_ + }\left(\mathscr{N}_1,K_1\right) \cap {Z_ + }\left(\mathscr{N}_2,K_2\right) \cap ... \cap {Z_ + }\left(\mathscr{N}_k,K_k\right) \subseteq  {Z_ + }\left(\mathscr{N},K\right)$.\\
			If the decomposition is incidence independent, then
			\item[ii.] ${Z_ + }\left( {\mathscr{N},K} \right) = {Z_ + }\left(\mathscr{N}_1,K_1\right) \cap {Z_ + }\left(\mathscr{N}_2,K_2\right) \cap ... \cap {Z_ + }\left(\mathscr{N}_k,K_k\right)$, and
			\item[iii.] ${Z_ + }\left( {\mathscr{N},K} \right) \ne \varnothing$ implies ${Z_ + }\left( {\mathscr{N}_i,K_i} \right) \ne \varnothing$ for each $i=1,...,k$.
		\end{itemize}
	\end{theorem}
	
	\section{Independent and incidence independent decompositions of chemical reaction networks}
	\label{section:results1}
	
	\subsection{Existence of independent and incidence independent decompositions}
	\label{subsection:independent:dec}
	
	We now focus our attention on a review of the concept of coordinate graph and the method of finding independent decompositions of CRNs, if they exist. The main reference for this section is the work of Hernandez and De la Cruz in \cite{hernandez:delacruz1}.
	
	We let ${ R}=\{{ R}_1,\ldots, { R}_m\}$ be a set of vectors such that ${\dim}(\text{span } { R})=\rho$ and $\{{ R}_1,\ldots,{ R}_\rho\}$ is linearly independent.
	The {\bf{coordinate graph}} of ${ R}$ is the undirected graph $G=(V,E)$ with vertex set $V=\{v_1,\ldots, v_\rho\}$ and edge set $E$ such that
	$(v_i,v_j) \in E$ if and only if there exists $k >\rho$ with ${ R}_k=\displaystyle \sum_{j=1}^\rho a_j { R}_j$ and both $a_i$ and $a_j$ are nonzero. 
	
	We introduce the following theorem which gives a necessary and sufficient condition for the existence of an independent decomposition of a set of vectors, and a method of finding a nontrivial independent decomposition, if it exists.
	\begin{theorem} \cite{hernandez:delacruz1} Let  ${R}$ be a finite set of vectors. An independent decomposition of ${R}$ exists if and only if 
		the coordinate graph of ${R}$ is not connected. 
		\label{nec:suf:inde}
	\end{theorem}
	
	The following steps are useful in finding an independent decomposition of a CRN, if it exists \cite{hernandez:delacruz1}.
	
	\begin{itemize}
		\item[1.] Get $N^T$, the transpose of the stoichiometric matrix $N$.
		\item[2.] Find a maximal linearly independent set of vectors say $\{R_{i_1},R_{i_2}\ldots, R_{i_\rho}\}$ that forms a basis for the row space of $N^T$.
		\item[3.] Construct the vertex set of the coordinate graph $G=(V,E)$ of $R$ by representing each $R_{i_j}$ as vertex $v_i$.
		\item[4.] For each vector $R_k$ distinct from 
		the elements of $\{R_{i_1},R_{i_2}\ldots, R_{i_\rho}\}$,
		write $R_k= \sum_j a_{k,j} R_{i_j}$.
		For each pair $a_{k,j_1}$ and $a_{k,j_2}$
		in the preceding sum,
		with both coefficients being nonzero, we add
		the edge $(v_{j_1},v_{j_2})$ to $E$.
		\item[5.] There is no nontrivial
		decomposition for $R$ if the formed coordinate graph $G$
		is connected. Otherwise, the
		reaction vectors corresponding to vertices
		belonging in the same connected component, together with the reaction vectors in their span
		constitute a partition of $R$ in the
		independent decomposition of
		$R$.
	\end{itemize}
	
	The next example gives us two things: (1) an illustration of getting a nontrivial independent decomposition, and (2) that if we have one equilibrium satisfying both subnetworks, it follows that this equilibrium is also an equilibrium of the given whole network.
	
	\begin{example}
		\label{mass:action:example}
		We consider the following CRN that satisfies mass action property:
		\begin{align*}
			{R_1}&:{2A} + {B} \to {A} + 2{B}\\
			{R_2}&:{2B} + {C} \to A+B+C\\
			{R_3}&:{2A} \to {A+C}\\
			{R_4}&:{B+C} \to {B+A}
		\end{align*}
		The corresponding set of ordinary differential equations is given by:
		\begin{align*}
			{\dfrac{{dA}}{{dt}}} &=  - {k_1}{A^2}B + {k_2}{B^2}C - {k_3}{A^2} + {k_4}BC\\
			{\dfrac{{dB}}{{dt}}} &=  {k_1}{A^2}B - {k_2}{B^2}C\\
			{\dfrac{{dC}}{{dt}}} &= {k_3}{A^2} - {k_4}BC
		\end{align*}
		We now get an nontrivial independent decomposition, if there is any. First, we obtain the transpose of the stoichiometric matrix.
		\[ N^T=\left[ {\begin{array}{*{20}{c}}
				{ - 1}&1&0\\
				1&{ - 1}&0\\
				{ - 1}&0&1\\
				1&0&{ - 1}
		\end{array}} \right]\]
		
		Let $R_i$ be the $i$th row of the matrix. A basis for the row space of the matrix is $\left\{ {{R_1},{R_3}} \right\}$. Then, we have
		\begin{itemize}
			\item[1.] $R_2=-R_1$, and
			\item[2.] $R_4=-R_3$.
		\end{itemize}
		
		With respect to getting the coordinate graph, let $R_1=v_1$ and $R_3=v_2$. Since we are forced to disconnect vertices $v_1$ and $v_2$, we obtain the following partition of the reaction set that gives a nontrivial independent decomposition of $R$:
		\begin{enumerate}
			\item $P_1=\{R_1,R_2\}$, and
			\item $P_2=\{R_3,R_4\}$.
		\end{enumerate}
		We associate independent subnetworks $\mathscr{N}_1$ and $\mathscr{N}_2$ of $\mathscr{N}$ to $P_1=\{R_1,R_2\}$ and $P_2=\{R_3,R_4\}$, respectively.
		
		Now, the corresponding set of ordinary differential equations for $\mathscr{N}_1$ is given by:
		\begin{align*}
			{\dfrac{{dA}}{{dt}}} &=  - {k_1}{A^2}B + {k_2}{B^2}C\\
			{\dfrac{{dB}}{{dt}}} &=  {k_1}{A^2}B - {k_2}{B^2}C
		\end{align*}
		such that the rate of concentration of species $C$ is not changing.
		On the other hand, the corresponding set of ordinary differential equations for $\mathscr{N}_2$ is given by:
		\begin{align*}
			{\dfrac{{dA}}{{dt}}} &=  -{k_3}{A^2} + {k_4}{B}C\\
			{\dfrac{{dC}}{{dt}}} &= {k_3}{A^2} - {k_4}BC
		\end{align*}
		such that the rate of concentration of species $B$ is not changing.
		Note that for particular rate constants 
		\begin{itemize}
			\item[1.] $k_1=1$ and $k_2=2$, $\left(A,B,C\right)=(2,1,2)$ is an equilibrium for $\mathscr{N}_1$, and
			\item[2.] $k_3=1$ and $k_4=2$, $\left(A,B,C\right)=(2,1,2)$ is an equilibrium for $\mathscr{N}_2$.
		\end{itemize}
		It follows that for particular rate constants $k_1=k_3=1$ and $k_2=k_4=2$, $\left(A,B,C\right)=(2,1,2)$ is an equilibrium for the whole network.
	\end{example}
	
	Let $R$ be the set of reaction vectors of a CRN. We recall that a decomposition of $R$ is a collection $\mathscr{P}=\{ P_1,\ldots, P_m\}$ of subsets of $R$ such that 
	\begin{center} $\bigcup P_i = R$ and $P_i \cap P_j = \varnothing $ for all $i \neq j$. 
	\end{center}
	Moreover, let $p_0$ be the dimension of the span of $R$ and $p_1=\displaystyle \sum_{i=1}^m \mbox{dim}(\mbox{span} P_i) $.
	Since $${\rm{span}} (U_1 \cup \ldots \cup U_k) = {\rm{span}} U_1 + \cdots + {\rm{span}} U_k,$$ and $\displaystyle \sum_{i=1}^k \dim V_i \leq \dim \left( \displaystyle \sum_{i=1}^k V_i \right)$,
	we have that $p_1$ is always less than or equal to $p_0$.
	If we have equality, then we say that the decomposition $\mathscr{P}$ is independent.
	
	Let $\mathscr{P}=\{ P_1,\ldots, P_m\}$ and $\mathscr{D}$ be a decomposition of $R$. We say that $\mathscr{D}$ is a simple refinement of $R$ if, upon reindexing if necessary,
	$$\mathscr{D}=\{ P_{1,1},\ldots, P_{1,k}, P_2,\ldots, P_m\},$$ where $P_1=P_{1,1}  \cup \cdots \cup P_{1,k}$. In the preceding, $\mathscr{P}$ is a simple coarsening of $\mathscr{D}$.
	Observe that if $\mathscr{D}$ is independent, then we have
	$$\begin{array}{rcl} \dim {\rm{span}} R &= &\left( \displaystyle \sum_{i=1}^k \dim {\rm{span}} P_{1,k} \right)+ \dim {\rm{span}} P_2+ \cdots + \dim {\rm{span}} P_m \\
		& \leq & \dim {\rm{span}}\left( \displaystyle \bigcup_{i=1}^k P_{1,i} \right) + \dim {\rm{span}} P_2 + \cdots + \dim {\rm{span}} P_m \\
		& \leq &  \dim {\rm{span}} P_1 + \dim {\rm{span}} P_2 + \cdots + \dim {\rm{span}} P_m \\ 
		& \leq & \dim {\rm{span}} R, \end{array}$$
	which tells us that the above inequalities are equalities, and so $\mathscr{P}$ is independent as well. Suppose that $\mathscr{P}$ is independent. Then $\mathscr{D}$
	is also independent if and only if
	$$\dim {\rm{span}} P_1 = \dim {\rm{span}} P_{1,1} + \cdots + \dim {\rm{span}} P_{1,k},$$
	that is, when ${\rm{span}} P_1 = {\rm{span}} P_{1,1} \oplus \cdots \oplus {\rm{span}} P_{1,k}$. Recall that a sum $V_1 + \cdots + V_k$ is direct \cite{Roman} if and only if 
	$V_i \cap \displaystyle \sum_{j \neq i} V_j =\{0\}$ for all $i=1,\ldots, k$, and so $\mathscr{D}$ is independent if and only if
	${\rm{span}} P_{1,i} \cap \displaystyle \sum_{j \neq i} {\rm{span}} P_{1,j} = \{0\}$.

	We have $\mathscr{D}$ is a refinement
	of $\mathscr{P}$ if there is a sequence of decompositions of $R$, say $\mathscr{P}_0:=\mathscr{P},\mathscr{P}_1,\ldots, \mathscr{P}_r:=\mathscr{D}$ such that  $\mathscr{P}_{i}$ is a simple
	refinement of $\mathscr{P}_{i-1}$ for $i=1,\ldots,r$. The above discussion implies that if $\mathscr{D}$ is independent, then $\mathscr{P}_i$  is independent for all $i$. To put it simply,
	if $\mathscr{P}$ has an independent refinement, then it has an independent simple refinement. We now prove the following theorem.
	
	\begin{theorem} Let  ${R}$ be a finite set of vectors. $R$ has exactly one independent (incidence independent) decomposition with no independent (incidence independent) refinement. We shall call this decomposition as the finest independent (incidence independent) decomposition of $R$.
		Moreover, the independent (incidence independent) decompositions of $R$ are precisely the finest decomposition of $R$ and its coarsenings.
		\label{lemma:possible:independent}
	\end{theorem}
	\begin{proof} We prove this theorem for the independent case. We can use an analogous proof for the incidence independence. Observe that if $R$ has only the trivial decomposition, then we are done. Suppose that $R$ has a nontrivial independent decomposition. We first prove the uniqueness of
		the finest independent decomposition of $R$. Let
		$\mathscr{P}_1=N_1 \cup \cdots \cup N_k$ and $\mathscr{P}_2=M_1 \cup \cdots \cup M_l$ be independent decompositions of $R$ having no further independent decomposition. 
		Consider $N_1$. Collect all sets in $\mathscr{P}_2$ that has nontrivial intersection with $N_1$. Reindexing when necessary, we can assume these sets are $M_1,\ldots, M_r$. We show that $r=1$, that
		is, $N_1=M_1$. One can then prove inductively that $k=l$ and that $N_i=M_i$ for all $i=1,\ldots, k$ as desired. Now, define $M_{1,1}= N_1 \cap M_1$ and $M_{1,2} = M_1 \setminus N_1$.
		If $M_{1,2} \neq \varnothing$, one checks that $M_{1,1} \subset {\rm{span}} N_1$ and $M_{1,2} \subset \displaystyle \sum_{i=2}^k {\rm{span}} N_i$, which implies that ${\rm{span}} M_{1,1} \leq {\rm{span}} N_1$ and
		${\rm{span}} M_{1,2} \leq  \displaystyle \sum_{i=2}^r {\rm{span}} N_i$. Since $ \displaystyle \sum_{i=1}^k {\rm{span}} N_i$ is direct, $${\rm{span}} M_{1,1} \cap {\rm{span}} M_{1,2} \subseteq  {\rm{span}} N_1 \cap  \displaystyle \sum_{i=2}^k {\rm{span}} N_i = \{0\}.$$
		Thus, the sum ${\rm{span}} M_{1,1} + {\rm{span}} M_{1,2} = {\rm{span}} M_1$ is direct. One checks that $M_{1,1} \cup M_{1,2} \cup M_2 \cup \cdots \cup M_l$ is an independent refinement
		of $\mathscr{P}_2$, which is a contradiction. Hence, $M_{1,2}= \varnothing$, that is, $M_{1} \subseteq N_1$. One can prove similarly that $M_i \subseteq N_1$ for $i=1,\ldots r$. 
		If $r>1$, we have that $N_1 = M_1 \cup \cdots \cup M_r$, since $\mathscr{P}_2$ is independent, $\dim {\rm{span}} N_1 = \dim {\rm{span}} M_1 + \cdots + \dim {\rm{span}} M_r$. 
		We now have that $M_1,\ldots, M_r, N_2,\ldots, N_k$ is an independent refinement of $\mathscr{P}_1$, which is a contradiction. Thus $k=1$ and that $M_1=N_1$, and this proves the uniqueness
		of the finest independent decomposition of $R$. The remaining part of the claim follows from this uniqueness and the fact that $R$ only has a finite number of elements
		and so an independent decomposition can be refined a finite number of times to obtain the finest independent decomposition of $R$.
	\end{proof}
	
	Suppose that $\mathscr{P}=N_1 \cup \cdots \cup N_k$ is the independent decomposition obtained from the 5-step method of finding independent decomposition. If $\mathscr{P}$ has an independent refinement,
	then $\mathscr{P}$ has an independent simple refinement, and we can assume that such a refinement is $(N_{1,1} \cup \cdots \cup N_{1,r}) \cup \cdots \cup N_k$, where $N_1$ is the disjoint union
	of $N_{1,1} \cup \cdots \cup N_{1,r}$. Suppose that $s$ is the dimension of ${\rm{span}} N_1$. As constructed in Algorithm \ref{algorithm:independent},  
	we can assume $N_1= \left\{ R_1,\ldots, R_s, \displaystyle \sum_{i=1}^s a_{1,i} R_i, \ldots \right\}$, where $R_1,\ldots, R_s$ is linearly independent, and $a_{i,j} \neq 0$ for all $i$ and $j$.
	We can reindex so that $ \displaystyle \sum_{i=1}^s a_{1,i}R_i \in N_{1,1}$. Assume that $t$ elements of $R_1,\ldots, R_s$ are in $N_{1,1}$ ($t$ may be zero), and after reindexing, we assume that
	these are $R_1,\ldots, R_t$. One checks that $$\left(\sum_{i=1}^s a_{1,i}R_i \right) - a_{1,1}R_1 - \cdots - a_{1,t} R_t = \sum_{i=t+1}^s a_{1,i}R_i \in  {\rm{span}} N_{1,1} \cap \sum_{i=2}^r N_{1,i},$$
	and that $\displaystyle \sum_{i=t+1}^s a_{1,i}R_i \neq 0$, which contradicts the fact that $(N_{1,1} \cup \cdots \cup N_{1,r}) \cup \cdots \cup N_k$ is independent. Thus, $R_1,\ldots, R_s \in N_{1,1}$. But we have
	$$
	\begin{array}{rcl}
		s&=&\dim {\rm{span}} N_1 \\ &=& \dim {\rm{span}} N_{1,1} + \cdots + \dim {\rm{span}} N_{1,r}  \\ &=& s +\dim {\rm{span}} N_{1,2} + \cdots + \dim {\rm{span}} N_{1,r} , \end{array}$$ and so
	there is a contradiction since each $N_{1,i}$ for $i>1$ contains a (nonzero) vector in $N_1$. Thus $\mathscr{P}$ has no independent refinement. This gives us the following remark.
	
	\begin{remark} The independent decomposition obtained from the method of finding an independent decomposition is precisely the finest independent decomposition of $R$.
		\label{remark:finest:independent} 
	\end{remark}
	
	Based on Theorem \ref{lemma:possible:independent} and Remark \ref{remark:finest:independent}, we modify the method of finding independent decomposition in the form of an algorithm given in Algorithm \ref{algorithm:independent}. It explicitly mentions that such independent decomposition is the finest.
	
	\begin{algorithm}
		\SetAlgoLined
		\bigskip
		{\bf{STEP 1}} {{\bf{Input}}: reaction network $\mathscr{N}$}\\
		\bigskip
		
		{\bf{STEP 2}} {Identify the transpose of the stoichiometric matrix, and correspondingly find a maximal linearly independent set of vectors as rows of the matrix}\\
		${N}^T=$ transpose of the stoichiometric matrix\\
		$B=$ maximal linearly independent set of vectors in the stoichiometric matrix $N^T$\\
		\bigskip
		
		{\bf{STEP 3}} {Identify vertices and edges of the coordinate graph $G = (V,E)$}\\
		\For{$q = 1$ to $|B|$}{
			$v_q:=R_q$\\
		}
		$V=\{v_q|q=1,\ldots,|B|\}$\\
		\For{$R_k \notin B$}{
			$R_k= \displaystyle \sum_j a_{k,j} R_{i_j}$ for elements $R_{i_j} \in B$\\
			\eIf{$a_{k,j_1} \ne 0$ and $a_{k,j_2}\ne 0$}{
				$(v_{j_1},v_{j_2}) \in E$
			}{
			}
		}
		\bigskip
		{\bf{STEP 4}} {{\bf{Output}:} The finest incidence independent decomposition with the information whether trivial or nontrivial}\\
		\eIf{$G$ is not connected}{
			there is no nontrivial incidence independent decomposition for $R$, i.e., $R$ itself induces the finest incidence independent decomposition with only one subnetwork
		}{
			reaction vectors $R_k$ corresponding to vertices
			belonging in the same connected component, together with the reaction vectors in their span
			constitute a partition of $R$ in the
			independent decomposition of
			$R$, i.e.,
			the collection of partitions of $R$ induced by $G = (V,E)$ is the finest incidence independent decomposition
		}
		\bigskip
		
		\caption{Method for finding the finest independent decomposition of a reaction network}
		\label{algorithm:independent}
	\end{algorithm}

	\begin{example}
		Consider the mathematical model for dynamics of coronavirus based on the target cell-limited model \cite{ciupe:heffernan,hernandez:vargas2,perelson} from Hernandez-Vargas and Velasco-Hernandez \cite{hernandez:vargas}.
		It has 3 variables which are time-dependent. Host cells can either be susceptible $(S)$ or infected $(I)$. Viral particles $(V)$ can infect susceptible cells of rate $\beta$, and cells release virus once they are productively infected of rate $p$, and $c$ is the clearing rate of the virus particles. Moreover, the infected cells die at the rate of $\delta$ \cite{hernandez:vargas}. 
		The reaction network $\mathscr{N}$ is given as follows:
		\begin{align*}
			{R_1}&:{U} + {V} \to {I} + {V}\\
			{R_2}&:{I} \to 0\\
			{R_3}&:I \to I+V\\
			{R_4}&:V \to 0
		\end{align*}
		We now use the method of finding an independent decomposition. First, we obtain the transpose of the stoichiometric matrix:
		\[{N}^T=\left[ {\begin{array}{*{20}{c}}
				{ - 1}&1&0\\
				0&{ - 1}&0\\
				{ 0}&0&1\\
				0&0&{ - 1}
		\end{array}} \right]\]
		
		Let $R_i$ be the $i$th row of the matrix. A basis for the row space of the matrix is $\left\{ {{R_1},{R_2},{R_3}} \right\}$. Then, we have $R_4=-R_3$.
		
		The coordinate graph, with $v_1=R_1$, $v_2=R_2$ and $v_3=R_3$ has three components yielding the following partition of the reaction set that gives a nontrivial independent decomposition of $R$:
		\begin{enumerate}
			\item $P_1= \{R_1\}$,
			\item $P_2= \{R_2\}$, and
			\item $P_3= \{R_3,R_4\}$.
		\end{enumerate}
		Note that the three components of the coordinate graph give rise to three subnetworks in the (finest) independent decomposition of $\mathscr{N}$.
		\label{example:hernandez:vargas}
	\end{example}

	\begin{algorithm}
		\SetAlgoLined
		\bigskip
		{\bf{STEP 1}} {{\bf{Input}}: reaction network $\mathscr{N}$}\\
		\bigskip
		
		{\bf{STEP 2}} {Determine the linkage classes $\mathscr{L}_\alpha$ and the reaction sets $R_\alpha$
			for $\alpha = 1,2, ... , \ell$}\\
		\bigskip
		
		{\bf{STEP 3}} {Identify the transpose of the incidence matrix per linkage class, and correspondingly find a maximal linearly independent set of vectors as rows of the transpose of the incidence matrix per linkage class}\\
		\For{$\alpha = 1$ to $\ell$}{
			${I_{a,\alpha}}^T=$ transpose of the incidence matrix restricted to linkage class $\alpha$\\
			$B_\alpha=$ maximal linearly independent set of vectors in the incidence matrix ${I_{a,\alpha}}^T$ which is restricted to linkage class $\alpha$\\
			proceed to STEP 4\\
			proceed to STEP 5
		}
		\bigskip
		{\bf{STEP 4}} {Identify vertices and edges of the coordinate graph $G_\alpha = (V_\alpha,E_\alpha)$}\\
		\For{$q = 1$ to $|B_\alpha|$}{
			$v_q:=R_q$\\
		}
		$V_\alpha=\{v_q|q=1,\ldots,|B_\alpha|\}$\\
		\For{$R_k \notin B_\alpha$}{
			$R_k= \displaystyle \sum_j a_{k,j} R_{i_j}$ for elements $R_{i_j} \in B_\alpha$\\
			\eIf{$a_{k,j_1} \ne 0$ and $a_{k,j_2}\ne 0$}{
				$(v_{j_1},v_{j_2}) \in E_\alpha$
			}{
			}
		}
		\bigskip
		{\bf{STEP 5}} {Check decomposition from coordinate graph of $G_\alpha = (V_\alpha,E_\alpha)$ per linkage class}\\
		\eIf{$G_\alpha$ is not connected}{
			there is no nontrivial incidence independent decomposition for $R_\alpha$
		}{
			reaction vectors $R_k$ corresponding to vertices
			belonging in the same connected component, together with the reaction vectors in their span
			constitute a partition of $R_\alpha$ in the
			incidence independent decomposition of
			$R_\alpha$
		}
		\bigskip
		{\bf{STEP 6}} {\bf{Output}}: {The finest incidence independent decomposition}\\
		The collection of partitions of $R_\alpha$ induced by $G_\alpha = (V_\alpha,E_\alpha)$ over all linkage classes is the finest incidence independent decomposition.\\	
		\bigskip
		\caption{Method for finding the finest incidence independent decomposition of a reaction network}
		\label{algorithm:incidence:independent}
	\end{algorithm}

	\begin{remark}
		The given main theorem and 5-step method of finding an independent decomposition of a CRN are also applicable in finding incidence independent decomposition, and thus giving an analogue for incidence independent decomposition. We just replace the transpose of the stoichiometric matrix by the transpose of the incidence matrix.
		There are differences between independence and incidence independence: the former is a stoichiometric property, while the latter is more closely related to the graph theoretic aspects. In particular, if there are at least two linkage classes, the linkage class decomposition is a nontrivial incidence independent decomposition. We incorporate this property in the construction of Algorithm \ref{algorithm:incidence:independent}, which 
		provides a method of finding the finest incidence independent decomposition of a reaction network.
		In addition, any set of subnetworks with just a single complex in common and whose reaction sets are pairwise disjoint, form an incidence independent decomposition, which is embedded in the algorithm. We use Example \ref{example:embedded} to illustrate this case.
		\label{remark:incidence:independent}
	\end{remark}
	
	\begin{example}
		Consider a CRN with reactions $A \to B$ and $B \to C$. Note that there is only one linkage class with just a single complex in common (i.e., complex $B$) as described in Remark \ref{remark:incidence:independent}. We have the following transpose of the incidence matrix:
		\[{I_a}^T=\left[ {\begin{array}{*{20}{c}}
				{ - 1}&1&0\\
				0&{ - 1}&1
		\end{array}} \right].\]	
		The matrix yields coordinate graph with two vertices yielding two subnetworks for the CRN: namely, $\{A \to B\}$ and $\{B \to C\}$, which correspond to an incidence independent decomposition, as expected from Remark \ref{remark:incidence:independent}.
		\label{example:embedded}
	\end{example}
	
	\begin{example}
		Consider Example \ref{example:hernandez:vargas}. The transpose of incidence matrix is as follows:
		\[{I_a}^T=\left[ {\begin{array}{*{20}{c}}
				{ - 1}&1&0&0&0\\
				0&{ 0}&-1&1&0\\
				{ 0}&1&-1&0&0\\
				0&0&{ 0}&1&-1
		\end{array}} \right]\]
		
		The network has only one linkage class. Let $R_i$ be the $i$th row of the matrix. A basis for the row space of the matrix is $\left\{ {{R_1},{R_2},{R_3},R_4} \right\}$.
		
		The coordinate graph, with vertices $v_1=R_1$, $v_2=R_2$, $v_3=R_3$, and $v_4=R_4$, has four components yielding the following partition of the reaction set that gives a nontrivial incidence independent decomposition of $R$:
		\begin{enumerate}
			\item $P_1= \{R_1\}$,
			\item $P_2= \{R_2\}$,
			\item $P_3= \{R_3\}$, and
			\item $P_4= \{R_4\}$.
		\end{enumerate}
		Note that the four components of the coordinate graph give rise to four subnetworks in the incidence independent decomposition of $\mathscr{N}$.
		\label{example:incidence:independent}
	\end{example}
	
	\subsection{Reaction networks with independent subnetwork of zero deficiency}
	\label{subsection:defzero:independent:subnetwork}
	
	We first consider the following Propositions \ref{def0:prop:feinberg1} and \ref{def0:prop:feinberg2}
	from \cite{feinberg12,feinberg}.
	After getting an independent decomposition, one may check if a network has an embedded deficiency zero subnetwork under independent decomposition, i.e., one of the subnetworks has deficiency zero and apply such propositions.
	
	\begin{proposition}
		\label{def0:prop:feinberg1}
		Consider a reaction system (not necessarily mass action) for which the underlying reaction network has zero deficiency. If $c^*$ is a steady
		state of the corresponding differential equations and $y_i$ is
		a complex in the network, then ${\rm supp \ } c^*$ contains ${\rm supp \ } y_i$
		only if $y_i$ is a member of a terminal strong linkage class.
		In addition, if ${\rm supp \ } c^*$ contains ${\rm supp \ } y_i$, then ${\rm supp \ } c^*$ also contains
		the support of each complex in the terminal strong
		linkage class to which $y_i$ belongs.
	\end{proposition}
	
	The proposition above implies that if the network has deficiency zero and if $c^*$ is a steady
	state, then ${\rm supp \ } c^*$ cannot contain the support of any
	complex that does not lie in a terminal strong linkage
	class. In other words, if $y_i$ is a complex that does not lie in a
	terminal strong linkage class then, at $c^*$,
	at least one species appearing in complex $y_i$ must have a
	zero component \cite{feinberg12}.
	
	The proposition can be viewed as a (slight) generalization of the classical result that any deficiency zero kinetic system with a positive equilibrium is weakly reversible, which is usually proved by combining Feinberg's result that any deficiency zero system with a positive equilibrium is complex balanced and Horn's result that any complex balanced system is weakly reversible. This is because if $c^*$ is a positive vector, ${\rm supp \ }c^*$ contains ${\rm supp \ }y$ for any complex. Hence, it follows that any complex is in a terminal strong linkage class, i.e., any linkage class is a terminal strong linkage class, which is equivalent to weak reversibility (s. \cite{feinberg} for more details). 
	
	\begin{proposition}
		\label{def0:prop:feinberg2}
		Suppose that for a given reaction system (not necessarily mass action), the ODE admits a steady state $c^*$. Suppose also that $y_i$ and $y_j$ are complexes in the underlying network such that there exists a directed reaction pathway leading from $y_i$ to $y_j$. If ${\rm supp \ } y_i$ is contained in ${\rm supp \ } c^*$ then ${\rm supp \ } y_j$ is also contained in ${\rm supp \ } c^*$. In particular, ${\rm supp \ } c^*$ either contains the supports of all complexes within a given strong linkage class or else it contains the support of none of them.		
	\end{proposition}
	
	In a deficiency zero system, it is shown in Fari\~nas et al. \cite{FAML2020} that the following holds:
	\begin{proposition}\label{ZDD_prop}
		Let $\mathscr{D}: \mathscr{N} = \mathscr{N}_1 \cup \ldots \cup \mathscr{N}_k$ be a decomposition of a zero deficiency kinetic system $(\mathscr{N}, K)$. Then the following statements are equivalent:
		\item[i.] $\mathscr{D}$ is incidence independent.
		\item[ii.] $\mathscr{D}$ is independent and ZDD (i.e., ``zero deficiency decomposition'' or all subnetworks have zero deficiency).
		\item[iii.] $\mathscr{D}$ is bi-independent.
	\end{proposition}
	In particular, incidence independence implies independence, but not conversely. 
	
	We now introduce the following easy proposition, which means that we can look at the finest independent decomposition and check if it has a subnetwork of deficiency zero.
	
	\begin{proposition}
		If the finest independent (incidence independent) decomposition of a reaction network has a deficiency zero subnetwork then it has an embedded deficiency zero independent (incidence independent) subnetwork.
	\end{proposition}
	
	\begin{proof}
		A deficiency zero subnetwork in the finest independent (incidence independent) decomposition is an embedded deficiency zero independent (incidence independent) subnetwork of the network.
	\end{proof}
	
	Note that an independent decomposition is not always incidence independent and hence, the finest independent decomposition may or may not be incidence independent.
	
	\begin{proposition}
		Suppose the finest bi-independent decomposition $\mathscr{D}$ of a network $\mathscr{N}$ (finest independent decomposition which is incidence independent) exists.
		Then, $\mathscr{N}$ has a deficiency zero subnetwork under decomposition $\mathscr{D}$ if and only if $\mathscr{N}$ has an embedded deficiency zero bi-independent subnetwork.
	\end{proposition}
	
	\begin{proof}
		Left to right direction is obvious. Conversely, take $\mathscr{N}_k$ as the embedded deficiency zero subnetwork. If the decomposition is not the finest, we decompose $\mathscr{N}_k=\mathscr{N}_{k_1} \cup \mathscr{N}_{k_2}$. Note that the deficiency $0=\delta _k = \delta _{k_1}+\delta _{k_2}=0+0$. We can repeat the process and one gets the finest for $\mathscr{N}_k$ with subnetwork of deficiency zero.
	\end{proof}
	
	\begin{table}
		\begin{center}
			\caption{Details for computation in Example \ref{ex:def0:ex}}
			\label{table:def0:ex}       
			\begin{tabular}{lcccc}
				\noalign{\smallskip}\hline\noalign{\smallskip}
				& reaction network  & rank & deficiency\\
				\noalign{\smallskip}\hline\noalign{\smallskip}
				\noalign{\smallskip}\hline\noalign{\smallskip}
				$\mathscr{N}$ & $B \to 2C$ & 2 & 1\\
				& $B+A \mathbin{\lower.3ex\hbox{$\buildrel\textstyle\rightarrow\over
						{\smash{\leftarrow}\vphantom{_{\vbox to.5ex{\vss}}}}$}} A$ & &\\
				& $B \to 0$ &  & \\
				\noalign{\smallskip}\hline\noalign{\smallskip}
				$\mathscr{N}_1$ & $B \to 2C$ & 1 & 0\\
				\noalign{\smallskip}\hline\noalign{\smallskip}
				$\mathscr{N}_2$		 & $B+A \mathbin{\lower.3ex\hbox{$\buildrel\textstyle\rightarrow\over
						{\smash{\leftarrow}\vphantom{_{\vbox to.5ex{\vss}}}}$}} A$ &1 &1\\
				& $B \to 0$ &  & \\
				\noalign{\smallskip}\hline\noalign{\smallskip}
			\end{tabular}
		\end{center}
	\end{table}
	
	\begin{example}
		\label{ex:def0:ex}	
		Consider the network $\mathscr{N}$ in Table \ref{table:def0:ex}. We can obtain from the method of finding an independent decomposition that $\mathscr{D}=\{\mathscr{N}_1,\mathscr{N}_2\}$ is such a decomposition. In particular, it is the finest one. $\mathscr{N}_1$ is an embedded independent subnetwork in $\mathscr{N}$ with deficiency zero. Using Proposition \ref{def0:prop:feinberg1}, we can deduce that if $c^*$ is steady state of $\mathscr{N}_1$ (in the set of nonnegative orthant of $\mathbb{R}^3$), the component of $c^*$ for the second position (i.e., concentration of $B$) is 0. By the Feinberg Decomposition Theorem, it is not possible that $c^*$ of $\mathscr{N}$ is positive, i.e., all components for the concentrations are positive. In addition, by inspecting $\mathscr{N}_2$ and by using Proposition \ref{def0:prop:feinberg2}, both $c_B$ and $c_A$ are zero, and hence, the only possible $c^*$ in the nonnegative orthant has the following form: $(0,0,0)$ or $(0,0,c_C)$ where $c_C>0$. We note that this result is not specific for mass action as stated in Propositions \ref{def0:prop:feinberg1} and \ref{def0:prop:feinberg2}.
	\end{example}

	\begin{example}
		We again consider Example \ref{example:hernandez:vargas} for the model of Hernandez-Vargas and Velasco-Hernandez with the following reaction network:
		\begin{align*}
			{R_1}&:{U} + {V} \to {I} + {V}\\
			{R_2}&:{I} \to 0\\
			{R_3}&:I \to I+V\\
			{R_4}&:V \to 0
		\end{align*}
		We already obtained the following independent subnetworks:
		$\mathscr{N}_1: \{R_1\}$,
		$\mathscr{N}_2: \{R_2\}$, and
		$\mathscr{N}_3: \{R_3,R_4\}$.
		Note that both $\mathscr{N}_1$ and $\mathscr{N}_2$ have deficiency zero. Suppose there is a steady state $c^*=(c_U,c_I,c_V)$ in the set of nonnegative orthant. By Proposition \ref{def0:prop:feinberg1}, from $\mathscr{N}_1$ we can deduce that $c_U$ or $c_V$ is zero. In addition, from $\mathscr{N}_2$, we have $c_I=0$. From $\mathscr{N}_3$ and by Proposition \ref{def0:prop:feinberg2}, both $c_I$ and $c_V$ are zero. These imply that the steady states of the (whole) reaction system (not necessarily mass action) in the set of nonnegative orthant is of the form $(0,0,0)$ or $(c_U,0,0)$ where $c_U>0$. A positive steady state (i.e., all components are positive) is not possible.
	\end{example}
	
	\subsection{Number of independent and incidence independent decompositions of a reaction network}
	\label{section:results2}
	
	We proceed with the following definition which is useful for formulations of the succeeding results.
	
	\begin{definition}
		Let $\mathscr{N}$ be a reaction network under a decomposition $\mathscr{D}$. The {\bf length} of $\mathscr{D}$ is the number of subnetworks of $\mathscr{N}$ under $\mathscr{D}$.
	\end{definition}
	
	Let $p$ and $q$ be the lengths of the finest independent decomposition and of the finest incidence independent decomposition, respectively. We shall count the number of such independent and incidence independent decompositions of a CRN, which we denote by $B_p$ and $B_q$, respectively.
	
	\begin{corollary}
		Let $\mathscr{N}$ be a CRN with $p,q \ge 2$ be the length of the finest independent and incidence independent decompositions, respectively. Then the number of independent and incidence independent decompositions of $\mathscr{N}$ are given by
		$${B_p} = \sum\limits_{k = 1}^p {\left( {\begin{array}{*{20}{c}}
					{p - 1}\\
					{k - 1}
			\end{array}} \right)} {B_{p - k}} {\text{ and }}{B_q} = \sum\limits_{k = 1}^q {\left( {\begin{array}{*{20}{c}}
					{q - 1}\\
					{k - 1}
			\end{array}} \right)} {B_{q - k}}, {\text{ respectively.}}$$
		\label{thm:number:independent}
	\end{corollary}
	
	\begin{proof}
		It follows from Theorem \ref{lemma:possible:independent}, that the number of independent decompositions with $p$ independent subnetworks is precisely the number of partitions of a set, i.e., nonempty and pairwise disjoint subsets of the set has union equals the set itself. This number is given by the Bell number. We prove the incidence independent decompositions in a similar manner. 
	\end{proof}
	\begin{proposition}
		\label{thm:bound:independent}
		Let $\mathscr{N}$ be a CRN with $r \ge 2$ reactions. If $P,Q$ are the number of independent and incidence independent decompositions of $\mathscr{N}$, then
		$$1 \le P,Q \le {B_r} = \sum\limits_{k = 1}^r {\left( {\begin{array}{*{20}{c}}
					{r - 1}\\
					{k - 1}
			\end{array}} \right)} {B_{r - k}}.$$
	\end{proposition}
	
	We can see in Corollary \ref{thm:number:independent} that the Bell number is a key to get the number of independent and incidence independent decompositions of a CRN. We define $B_0=1$, and we can compute recursively $B_1=1$, $B_2=2$, $B_3=5$, and so on.
	In Proposition \ref{thm:bound:independent}, the lower bound is attained when there is only the trivial independent (incidence independent) decomposition, and the upper bound is attained when there are $r$ subnetworks in the finest independent (incidence independent) decomposition.
	For further discussion on Bell numbers, one may refer to \cite{bell:reference}.
	
	\begin{example}
		\label{ex:dec:created}
		Consider the CRN with the following reactions.
		\[\begin{array}{l}
			{R_1}:{X_1} + {X_5} \to {X_2} + {X_5}\\
			{R_2}:{X_2} \to 0\\
			{R_3}:{X_1} \to {X_1} + {X_4}\\
			{R_4}:{X_3} \to {X_4}\\
			{R_5}:{X_5} \to 0
		\end{array}\]
		We can verify that the finest independent decomposition of the CRN is given by $\mathscr{N}_1 \cup \mathscr{N}_2 \cup \mathscr{N}_3$ where $\mathscr{N}_1 =\{ R_1,R_2 \}$, $\mathscr{N}_2 =\{ R_3,R_4 \}$, and $\mathscr{N}_3 =\{ R_5 \}$. There are $B_3=5$ independent decompositions of $\mathscr{N}$. The list of such decompositions is provided in Table \ref{tab:list:ind:dec}. We can also check that there are 5 subnetworks in the finest incidence independent decomposition of $\mathscr{N}$ and hence, there are $B_5=52$ such decompositions. In this case, the upper bound for Proposition \ref{thm:bound:independent} is attained.
	\end{example}
	
	\begin{table}
		\caption{List of all independent decompositions of $\mathscr{N}$ in Example \ref{ex:dec:created}} 
		\centering 
		\begin{tabular}{llll} 
			\hline 
			Finest & The Rest of Nontrivial & Trivial \\
			Decomposition& Independent Decomposition & Decomposition\\
			\hline 
			$\{ \mathscr{N}_1, \mathscr{N}_2, \mathscr{N}_3 \}$ & $\{ \mathscr{N}_1 \cup \mathscr{N}_2, \mathscr{N}_3 \}$ & $\{\mathscr{N}\}$ \\ %
			& $\{ \mathscr{N}_1, \mathscr{N}_2 \cup \mathscr{N}_3 \}$ &  \\
			& $\{ \mathscr{N}_1 \cup \mathscr{N}_3, \mathscr{N}_2 \}$ &  \\%
			\hline 
		\end{tabular}
		\label{tab:list:ind:dec}
	\end{table}
	
	\section{Connections between weakly reversible decompositions, and independent and incidence independent decompositions}
	\label{section:results3}
	
	Komatsu and Nakajima \cite{komatsu} provided an algorithm for finding a weakly reversible decomposition of a CRN. On the other hand, Hernandez and Mendoza \cite{HM2021:WRDCKS} established an algorithm for finding a weakly reversible decomposition of a non-complex factorizable (NF) chemical kinetic system, that considers both the underlying CRN and kinetics, into weakly reversible complex factorizable (CF) subsystems.
	
	In our work, we focus on significant relationships between weakly reversible decomposition of a CRN, and its independent and incidence independent decompositions. 
	
	Weakly reversible networks are clearly the most important class of reaction networks, and it is good to observe an invariance of the property with respect to decompositions, which is reflected in the following well-known Proposition:
	
	\begin{proposition}\label{prop:WRcovering}
		A network is weakly reversible if and only if it has a weakly reversible covering.
	\end{proposition}
	
	\begin{proof}
		Let $\mathscr N$ be a weakly reversible network. Consider the linkage class decomposition $\mathscr N=\mathscr N_1\cup \cdots \cup \mathscr N_{\ell}$. Clearly, each subnetwork in this decomposition is weakly reversible. This decomposition serves as a weakly reversible covering of $\mathscr N$. For the converse, if a network has a weakly reversible covering, any reaction is in one of the subnetworks and by definition, it is contained in a directed cycle there. It follows that the parent network is weakly reversible. 
	\end{proof}
	
	Proposition \ref{prop:WRcovering} accepts the trivial decomposition in the case of a single linkage class. Otherwise, one must restrict the validity to the case $l > 1$.
	
	\subsection{Incidence independence is a sufficient condition for a decomposition's weak reversibility}
	
	Incidence independence also enforces the weak reversibility property, but this is not a surprise given the close relationship of the incidence map with weak reversibility.
	The proposition extends the result that in a weakly reversible network, the linkage class decomposition is a weakly reversible decomposition. Hence, even in the single linkage class case, non-trivial weakly reversible decompositions can exist.
	
	\begin{proposition}\label{2.13}
		Let $\mathscr N$ be a weakly reversible network. If $\mathscr N=\mathscr N_1\cup \cdots \cup \mathscr N_k$ is an incidence independent decomposition,  then it is also weakly reversible.
	\end{proposition}
	
	\begin{proof}
		Since $\mathscr N$ is weakly reversible, ${\rm Ker \ } I_{a}$ contains a positive vector $x$. We can write $x=x_1+x_2+\cdots+x_k$ with $x_i\in \mathbb R^{\mathscr R_i}$. Let $I_{a,i}$ be the restriction of $I_a$ to $\mathbb R^{\mathscr R_i}$. Since $I_a$ is a linear map, $I_a(x)=I_{a,1}(x)+I_{a,2}(x)+\cdots+I_{a,k}(x)$. Hence, $0+0+\cdots+0=0=I_{a,1}(x)+I_{a,2}(x)+\cdots+I_{a,k}(x)$. Because the decomposition is assumed to be incidence independent, it follows that $I_{a,i}(x)=0$ for each $i$, i.e., there is a positive vector in ${\rm Ker \ } I_{a,i}$, which is equivalent to the weak reversibility of $\mathscr N_i$. 
	\end{proof}
	
	\begin{figure}
		\begin{center}
			\includegraphics[width=10cm,height=3cm,keepaspectratio]{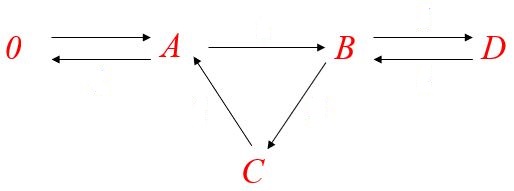}
			\caption{CRN in Example \ref{ex:networkWRexample2}
			} 
			\label{fig:networkWRexample}
		\end{center}
	\end{figure}
	
	\begin{figure}[ht]
		\begin{center}
			\includegraphics[width=10cm,height=3cm,keepaspectratio]{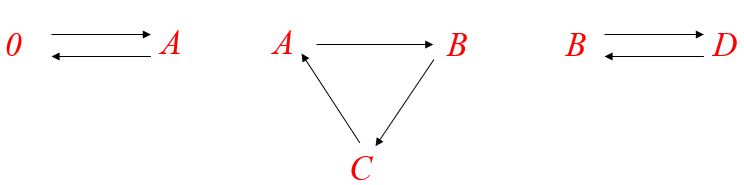}
			\caption{Weakly reversible decomposition of CRN in Figure \ref{fig:networkWRexample}}
			\label{fig:networkWRexampleDEC2}
		\end{center}
	\end{figure}
	
	\begin{example}
		We refer to the CRN in Figure \ref{fig:networkWRexample} and we decompose it into three subnetworks as provided in Figure \ref{fig:networkWRexampleDEC2}. Since $(n_1-l_1)+(n_2-l_2)+(n_3-l_3)=1+2+1=4$, and $n-l=4$, then the decomposition is incidence independent, and hence, weakly reversible.
		\label{ex:networkWRexample2}
	\end{example}
	
	\begin{figure}
		\begin{center}
			\includegraphics[width=17cm,height=4cm,keepaspectratio]{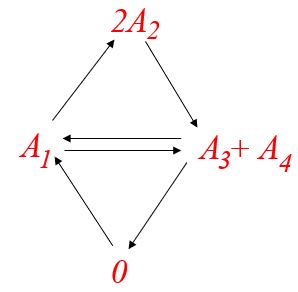}
			\caption{Reaction Network in Example \ref{ex:counterexample:inciinde:WR} adapted from \cite{marton}}
			\label{fig:counterexample:inde:WR}
		\end{center}
	\end{figure}
	
	\begin{example}
		Consider the network $\mathscr{N}$ given in Figure \ref{fig:counterexample:inde:WR} from \cite{marton}.
		Let $\mathscr{N}=\mathscr{N}_1\cup \mathscr{N}_2$, where $\mathscr{N}_1$ and $\mathscr{N}_2$ are the upper and the lower directed cycles, respectively. 
		Note that the $n_1-l_1=n_2-l_2=2$, but $n-l=3$. It follows that the induced decomposition is not incidence independent but it is weakly reversible.
		\label{ex:counterexample:inciinde:WR}	
	\end{example}
	
	\subsection{Independence and weak reversibility of decompositions}
	
	The following example shows that, in general, in a weakly reversible network, an independent decomposition need not be weakly reversible:
	
	\begin{example}
		Let $\mathscr{N} = \mathscr{N}_1 \cup \mathscr{N}_2$ where
		$\mathscr{N}_1: 0 \to X, X+Y \to Y$ and\\
		$\mathscr{N}_2: X \to X+Y, Y \to 0$.
	\end{example}
	
	The network is cyclic, and hence weakly reversible. Now, $s = 2$ and $s_1 = s_2 =1$, hence the decomposition is independent, but clearly not weakly reversible. The network is actually the total realization of an S-system with two dependent variables and the decomposition is its species decomposition, which is always independent \cite{AJLM2017,arceo2015,FAML2020}.
	
	Due to Proposition \ref{ZDD_prop} and Proposition \ref{2.13}, some weakly reversible networks can have independent decompositions that are also weakly reversible. Examples of which are zero deficiency networks with decompositions that are ZDD. 
	
	\begin{proposition} \label{sumdelta}
		Let $\mathscr{N}$ be a weakly reversible CRN where for every decomposition $\mathscr{N} = \mathscr{N}_1 \cup \ldots \cup \mathscr{N}_k$, $\delta = \delta _1 + \ldots + \delta _k$. Then any independent decomposition is weakly reversible.
	\end{proposition}
	
	\begin{proof}
		Since the decomposition of weakly reversible $\mathscr{N}$ into $k$ subnetworks satisfies $\delta = \delta _1 + \ldots + \delta _k$, then any independent decomposition is incidence independent, and vice versa.
		
	\end{proof}
	
	\begin{example}
		Note that monospecies CRNs (i.e., CRNs with species that serve as complexes) always have zero deficiencies. Hence, to get a weakly reversible decomposition of a zero deficiency monospecies CRN, we just have make sure that the subnetworks are also of zero deficiencies.
	\end{example}
	
	Proposition \ref{sumdelta} raises the following interesting question: Besides the weakly reversible monospecies networks, which other types of networks are in this class?
	
	Another class of weakly reversible networks is given by the reversible networks:	
	
	\begin{proposition} \label{2.14}
		Let $\mathscr N$ be a reversible network. If $\mathscr N=\mathscr N_1\cup \cdots \cup \mathscr N_k$ is an independent decomposition, then it is also reversible. 
	\end{proposition}
	
	\begin{proof}
		Let $\mathscr N=\mathscr N_1\cup \cdots \cup \mathscr N_k$ be an independent decomposition. We want to show that this decomposition is reversible. Consider the subnetwork $\mathscr N_i$. Let $R:y'\rightarrow y''$ be a reaction in $\mathscr N_i$. Then, the corresponding nonzero reaction vector $y''-y'$ is contained in the stoichiometric subspace $S_i$ of $\mathscr N_i$. Since it is assumed that the parent network $\mathscr N$ is reversible, the reaction $R^*: y''\rightarrow y'$ exists in  $\mathscr N$. If $R^*$ is a reaction in another subnetwork say $\mathscr N_j$, then the corresponding vector $y'-y''$ is contained in the stoichiometric subspace $S_j$ of $\mathscr N_j$. It follows that $y''-y'\in S_j$ suggesting that $S_i\cap S_j\neq \{0\}$. This is a contradiction since the decomposition is assumed to be independent. Hence, $R^*$ occurs in $\mathscr N_i$ implying that $\mathscr N_i$ is reversible. It follows that the decomposition is also reversible. 
	\end{proof}
	
	We emphasize that the contrapositive of Proposition \ref{2.14} is useful in determining the non-independence of a decomposition of a reversible network. Compared to the process of knowing that a decomposition is independent, in most cases, it is easy to check if the subnetworks are reversible or not.

	\section{Summary and Outlook}
	\label{section:summary}
	In this section, we provide a summary and outlook of our work.
	\begin{itemize}
		\item[1.] We revisited the method of Hernandez and De la Cruz \cite{hernandez:delacruz1} of finding an independent decomposition of a CRN. We found out that the output decomposition is the finest one. We also provided some results for networks with embedded deficiency zero subnetwork.
		\item[2.] We established an algorithm of finding the finest incidence independent decomposition of a CRN.
		\item[3.] We determined all the forms of independent and incidence independent decompositions.
		\item[4.] We provided the number of independent and incidence independent decompositions of a CRN.
		\item[5.] We found out that for weakly reversible networks, incidence independent is a sufficient condition for weak reversibility of a decomposition, and we identified subclasses of weakly reversible networks where any independent decomposition is weakly reversible.
		\item[6.] One may apply the theory provided in this work in different examples of networks in literature.
		\item[7.] More conditions or connections among independent, incidence independent, and weakly reversible decompositions can be explored.
	\end{itemize}
	
	\section*{Acknowledgement}
		This work was funded by the UP System Enhanced Creative Work and Research Grant (ECWRG-2020-2-11R).

\end{document}